\documentclass[11pt]{article}

\usepackage[a4paper]{anysize}\marginsize{2.5cm}{2.5cm}{1.3cm}{2cm}
\pdfpagewidth=\paperwidth \pdfpageheight=\paperheight
\usepackage{amsfonts,amssymb,amsthm,amsmath,eucal}
\usepackage{pgf}
\usepackage{bbm}
\usepackage{tikz} %Used for drawing picture in LaTeX
\usepackage{mathrsfs}
\usetikzlibrary{arrows}
\usepackage{color}
\usepackage{float}
\usepackage{subfigure}
\usepackage{caption}
\usepackage{graphicx}
\usepackage{url}
\usepackage[most]{tcolorbox}
\usepackage{hyperref}
\usepackage{cleveref}

\definecolor{qqqqff}{rgb}{0.,0.,0.}

\theoremstyle{plain}
\newtheorem{thm}{Theorem}[section]
\newtheorem{theorem}[thm]{Theorem}

\newtheorem{proposition}[thm]{Proposition}
\newtheorem{corollary}[thm]{Corollary}

\theoremstyle{definition}
\newtheorem{definition}[thm]{Definition}
\newtheorem{remark}[thm]{Remark}

\newtheorem{question}[thm]{Question}

\newtheorem{thevarthm}[thm]{\varthmname}

\newenvironment{varthm*}[1]{\trivlist\item[]{\bf #1.}\it}{\endtrivlist}

\renewcommand\leq{\leqslant}

\newcommand\be{\begin{eqnarray*}}
\newcommand\ee{\end{eqnarray*}}

\newcommand\PP{\mathbb P}

\newcommand\newop[2]{\def#1{\mathop{\rm #2}\nolimits}}
\newop\log{log}
\newop\ord{ord}
\newop\Gal{Gal}
\newop\SL{SL}
\newop\Bl{Bl}
\newop\mult{mult}
\newop\mass{mass}
\newop\div{div}
\newop\codim{codim}
\newop\sing{sing}
\newop\Zeroes{Zeroes}

\def\keywordname{{\bfseries Keywords}}%
\def\keywords#1{\par\addvspace\medskipamount{\rightskip=0pt plus1cm
\def\and{\ifhmode\unskip\nobreak\fi\ $\cdot$
}\noindent\keywordname\enspace\ignorespaces#1\par}}
\def\subclassname{{\bfseries Mathematics Subject Classification
(2020)}\enspace}
\def\subclass#1{\par\addvspace\medskipamount{\rightskip=0pt plus1cm
\def\and{\ifhmode\unskip\nobreak\fi\ $\cdot$
}\noindent\subclassname\ignorespaces#1\par}}

\begin{document}

\title{From a Pascal construction to the Burkhardt quartic}
\author{Tomasz Szemberg and Justyna Szpond}
\date{}
\maketitle
\thispagestyle{empty}
\parskip=.05cm

\begin{abstract}
We continue the study of Pascal-type residual constructions in projective
four-space. Starting from two $k$-tuples of hyperplanes in $\PP^4$ such that
the $k$ diagonal intersection planes are contained in a hyperplane, one obtains
a residual hypersurface of degree $k-1$ containing the remaining $k^2-k$
planes. In this work we consider the case $k=5$, where the twenty residual
planes are contained in a quartic threefold. A balanced
specialization of this construction is projectively equivalent to the celebrated
Burkhardt quartic.

In this model the twenty residual planes form one half of the forty Jacobi
planes on the Burkhardt quartic. We reveal their incidence structure as governed by the
directed complete graph on five vertices. The forty nodes naturally forced by
these planes split as $30+10$, and the Burkhardt specialization adds five
further nodes. We also write down the complementary twenty Jacobi planes explicitly and
describe all forty Steiner hyperplanes in Pascal coordinates.

\subclass{Primary 
14N20; %Configurations and arrangements of linear subspaces
Secondary 
14N05, %Projective techniques in algebraic geometry
14N25, %Varieties of low degree
14J30, %Threefolds
14J70, 
51A20.
}
\end{abstract}

\section{Introduction}
Classical incidence theorems often have residual interpretations. Pascal's
theorem and the Braikenridge--Maclaurin theorem are familiar examples: a
special incidence condition forces the residual part of a complete
intersection to lie on a curve of smaller than expected degree. Higher-dimensional analogues of this idea lead to configurations of linear spaces on
special hypersurfaces.

The case immediately preceding the one studied here was considered in
\cite{PokoraSzemberg2026Pascal}. There, a Pascal-type construction in $\PP^4$
with $k=4$ produces a residual cubic threefold. In general position this cubic
is the Segre cubic, and the construction recovers the fifteen planes on it
together with the Cremona--Richmond configuration. That work was motivated in
part by residual forms of Pascal-type theorems, including the spatial Pascal
theorem for lines in $\PP^3$ due to Le Van \cite{LeVan2026}.

The purpose of the present note is to study the next case, namely $k=5$. The
same residual construction now gives a quartic threefold containing twenty
residual planes. At first sight this might look like merely the next member of
a formal sequence. The main point of the paper is that, for a distinguished
balanced choice of the data, this quartic is not new: it is the classical
Burkhardt quartic.

It is useful to explain first how many parameters the construction has. After
normalizing the data, the quartic equation can be written in an auxiliary
projective space $\PP^5$ with homogeneous coordinates
$$
a,u_1,\ldots,u_5
$$
as
$$
\mathcal X:
\quad
a^4+a^2e_2(u_1,\ldots,u_5)+e_4(u_1,\ldots,u_5)=0,
$$
where $e_2$ and $e_4$ are the elementary symmetric functions in the five
variables $u_1,\ldots,u_5$.
A Pascal quartic threefold is obtained by taking a hyperplane section
$$
X_\lambda=\mathcal X\cap L_\lambda,
$$
where
$$
L_\lambda:
\lambda_0a+\lambda_1u_1+\cdots+\lambda_5u_5=0.
$$
Thus the natural parameter space of these linear sections is the dual
projective space $(\PP^5)^\vee$. This is not yet a moduli space, since we do
not quotient by projective equivalence; it is the parameter space directly
attached to the Pascal construction.

Among these hyperplane sections there is a distinguished line, the balanced
line. It is defined by the condition that the coefficients of
$u_1,\ldots,u_5$ in the linear relation are all equal. Hence, after rescaling,
it can be written as
$$
u_1+\cdots+u_5=5ma.
$$
Equivalently, the balanced members are parametrized by the projective line
$$
[\lambda_0:\lambda:\lambda:\lambda:\lambda:\lambda]\subset (\PP^5)^\vee.
$$
The Burkhardt quartic corresponds to the distinguished point on this line
specified by the condition
$$
m^2=-\frac{3}{25}.
$$
If $\eta^2=-3$, this point is represented by the relation
$$
u_1+\cdots+u_5=\eta a.
$$
For this member, a linear change of coordinates identifies the corresponding
quartic with the standard Burkhardt model
$$
\mathcal B:
\quad
\left\{
\begin{array}{ccc}
\sigma_1(x_0,\ldots,x_5) & = & 0 \\
\sigma_4(x_0,\ldots,x_5) & = & 0
\end{array}\right.
\subset \PP^5,
$$
where $\sigma_1$ and $\sigma_4$ are now elementary symmetric functions in the
six variables $x_0,\ldots,x_5$.

We recall that the Burkhardt quartic is the unique quartic threefold with
forty-five nodes, up to projective equivalence; see
\cite{Todd1936,deJongShepherdBarronVanDeVen1990,CheltsovTschinkelZhang2025}.
Its classical geometry includes forty Jacobi planes and forty Steiner
hyperplanes; see also \cite{Burkhardt1891,BruinNasserden2018,CheltsovTschinkelZhang2025}.
It is worth to point out that the Burkhardt quartic is a classical modular threefold. It is closely related
to the Siegel modular threefold with the full level-three structure. We refer to
Hunt's monograph~\cite{Hunt1996} for a detailed discussion of this classical
and modular background.

The Pascal construction gives a new way to see a large part of this classical
configuration. The twenty residual planes become one half of the forty Jacobi
planes on the Burkhardt quartic. They are naturally indexed by the directed
edges of the complete graph on five vertices. Their incidences force forty
nodes, split into two intrinsic classes of cardinalities
$$
30+10.
$$
The Burkhardt specialization adds five further nodes, giving the classical
number
$$
45=30+10+5.
$$
Each of the twenty Pascal planes contains exactly nine of these nodes, as a
Jacobi plane should. The complementary twenty Jacobi planes also admit simple
explicit equations in the same Pascal coordinates; they are written using a
primitive third root of unity. Finally, the Steiner hyperplanes can likewise
be described directly in Pascal coordinates. In particular, the original
hyperplanes from the Pascal construction already give ten of them.

The paper is organized as follows. In Section~\ref{sec:pascal-quartic} we
recall the residual construction and introduce the balanced Pascal quartic. In
Section~\ref{sec:burkhardt-identification} we identify this balanced member
with the Burkhardt quartic. In Section~\ref{sec:pascal-planes} we describe the
configuration of the twenty Pascal planes. In Section~\ref{sec:nodes} we write
down the forty-five nodes and compute their incidences with the Pascal planes.
In Section~\ref{sec:complementary-planes} we give explicit equations for the
complementary twenty Jacobi planes. Finally, in Section~\ref{sec:steiner} we
describe the Steiner hyperplanes in Pascal coordinates.

\section{The quartic Pascal threefold}
\label{sec:pascal-quartic}

We first recall the residual construction in the case $k=5$.

Let
$$
F_1,\ldots,F_5,
\qquad
G_1,\ldots,G_5
$$
be two quintuples of hyperplanes in $\PP^4$, with equations
$$
F_i=(f_i=0),
\qquad
G_j=(g_j=0).
$$
For $1\leq i,j\leq 5$ put
$$
\Pi_{ij}=F_i\cap G_j.
$$
We assume that the five diagonal planes
$$
\Pi_{11},\Pi_{22},\Pi_{33},\Pi_{44},\Pi_{55}
$$
are contained in a hyperplane
$$
H=(h=0).
$$
After rescaling the equations of the hyperplanes, we may assume that
$$
h=f_i+g_i,
\qquad i=1,\ldots,5.
$$

\begin{definition}
The quartic threefold
$$
X=
\left(
\frac{f_1f_2f_3f_4f_5+g_1g_2g_3g_4g_5}{h}=0
\right)
\subset \PP^4
$$
will be called the quartic Pascal threefold associated to the above data.
\end{definition}

\begin{proposition}\label{prop:pascal-planes-contained}
The quartic Pascal threefold $X$ contains the twenty residual planes
$$
\Pi_{ij}=F_i\cap G_j,
\qquad i\neq j.
$$
\end{proposition}

\begin{proof}
On $\Pi_{ij}$ with $i\neq j$ one has $f_i=0$ and $g_j=0$. Hence both products
$$
f_1f_2f_3f_4f_5
\qquad\text{and}\qquad
g_1g_2g_3g_4g_5
$$
vanish on $\Pi_{ij}$. Since $\Pi_{ij}$ is not one of the diagonal planes
contained in $H$, the quotient vanishes identically on $\Pi_{ij}$.
\end{proof}

Writing $g_i=h-f_i$ and denoting by $e_r=e_r(f_1,\ldots,f_5)$ the elementary
symmetric polynomial of degree $r$, we get
$$
X:\quad e_4-he_3+h^2e_2-h^3e_1+h^4=0.
$$
Equivalently, we put
$$
a=\frac h2,
\qquad
u_i=f_i-\frac h2=\frac{f_i-g_i}{2}.
$$
Then
$$
f_i=a+u_i,
\qquad
g_i=a-u_i.
$$
Since $a,u_1,\ldots,u_5$ are six linear forms on $\PP^4$, they satisfy one
linear relation. Thus there are scalars $\lambda_0,\lambda_1,\ldots,\lambda_5$,
unique up to a common factor, such that
$$
\lambda_0 a+\lambda_1u_1+\cdots+\lambda_5u_5=0.
$$
The quartic Pascal threefold is therefore the hypersurface
$$
X:\quad a^4+a^2e_2(u_1,\ldots,u_5)+e_4(u_1,\ldots,u_5)=0
$$
inside the hyperplane
$$
\lambda_0 a+\lambda_1u_1+\cdots+\lambda_5u_5=0
$$
of the auxiliary space with coordinates $a,u_1,\ldots,u_5$.

The balanced case is the case in which the five coefficients of
$u_1,\ldots,u_5$ are equal. After rescaling the relation, we may write it as
$$
u_1+\cdots+u_5=5ma.
$$
The Burkhardt specialization corresponds to
$$
m^2=-\frac{3}{25}.
$$
Equivalently, if $\eta^2=-3$, then
$$
u_1+\cdots+u_5=\eta a.
$$
This motivates the following definition.

\begin{definition}\label{def:balanced}
Let $\eta$ be a square root of $-3$. 

The balanced Pascal quartic is the
quartic threefold
$$
X_\eta:\quad a^4+a^2e_2(u_1,\ldots,u_5)+e_4(u_1,\ldots,u_5)=0
$$
in the hyperplane
$$
L_\eta:\quad u_1+\cdots+u_5=\eta a.
$$
\end{definition}

\begin{remark}
In the balanced case, in the original variables,
$$
\sum_{i=1}^5(f_i-g_i)=\pm i\sqrt 3\,h.
$$
\end{remark}

\section{Identification with the Burkhardt quartic}
\label{sec:burkhardt-identification}

In this section we identify the balanced Pascal quartic with the classical Burkhardt
quartic. We use the model
$$
\mathcal B:
\quad
\left\{
\begin{array}{ccc}
\sigma_1(x_0,\ldots,x_5) & = & 0 \\
\sigma_4(x_0,\ldots,x_5) & = & 0
\end{array}\right.
\subset \PP^5,
$$
This threefold is the Burkhardt quartic; see, for example,
\cite{Burkhardt1891,Todd1936,deJongShepherdBarronVanDeVen1990}.

\begin{theorem}\label{thm:balanced-is-burkhardt}
The balanced Pascal quartic $X_\eta$ is projectively equivalent to the
Burkhardt quartic $\mathcal B$.
\end{theorem}

\begin{proof}
Define six linear forms by
$$
x_0=\frac{2\eta}{3}a,
\qquad
x_i=u_i-\frac{\eta}{3}a,
\quad i=1,\ldots,5.
$$
Since $u_1+\cdots+u_5=\eta a$, we have
$$
x_0+x_1+\cdots+x_5=0.
$$
Put
$$
b=\frac{\eta}{3}a.
$$
Then
$$
x_0=2b,\qquad x_i=u_i-b,\quad i=1,\ldots,5,
$$
and the relation $u_1+\cdots+u_5=\eta a$ becomes $e_1(u_1,\ldots,u_5)=3b$.
%Let $y_i=u_i-b$. 
Then
$$
\sigma_4(x_0,\ldots,x_5)=e_4(x_1,\ldots,x_5)+2b\,e_3(x_1,\ldots,x_5).
$$
We have
$$
e_3(x_1,\ldots,x_5)=e_3(u_1,\ldots,u_5)-3be_2(u_1,\ldots,u_5)+6b^2e_1(u_1,\ldots,u_5)-10b^3
$$
and
$$
e_4(x_1,\ldots,x_5)=e_4(u_1,\ldots,u_5)-2be_3(u_1\ldots,u_5)+3b^2e_2(u_1,\ldots,u_5)-4b^3e_1(u_1,\ldots,u_5)+5b^4.
$$
Therefore 
$$
\begin{aligned}
\sigma_4(x_0,\ldots,x_5)
&=
e_4(u_1,\ldots,u_5)-3b^2e_2(u_1,\ldots,u_5)+8b^3e_1(u_1,\ldots,u_5)-15b^4  \\
&=
e_4(u_1,\ldots,u_5)-3b^2e_2(u_1,\ldots,u_5)+9b^4,
\end{aligned}
$$
because $e_1(u_1,\ldots,u_5)=3b$. Since $\eta^2=-3$, we have
$$
b^2=-\frac{1}{3}a^2
\qquad\text{and}\qquad
9b^4=a^4.
$$
Thus
$$
\sigma_4(x_0,\ldots,x_5)
=
a^4+a^2e_2(u_1,\ldots,u_5)+e_4(u_1,\ldots,u_5).
$$
Thus the equation of $X_\eta$ is transformed into $\sigma_4=0$ on the
hyperplane $\sigma_1=0$. This proves the claim.
\end{proof}

\begin{remark}
The family
$$
\sum_{i=0}^5x_i=0,
\qquad
\sum_{i=0}^5x_i^4-t\left(\sum_{i=0}^5x_i^2\right)^2=0
$$
contains several special nodal quartic threefolds studied by van der Geer and
by Cheltsov--Shramov; see \cite{vanDerGeer1982,CheltsovShramov2016}. 
Let $p_r=\sum_{i=0}^5x_i^r$ and let $\sigma_r$ denote the elementary
symmetric polynomial of degree $r$ in $x_0,\ldots,x_5$. Newton identities give
$$
p_2-\sigma_1p_1+2\sigma_2=0
$$
and
$$
p_4-\sigma_1p_3+\sigma_2p_2-\sigma_3p_1+4\sigma_4=0.
$$

On the hyperplane $\sigma_1=0$ we have $p_1=0$, hence
$$
p_2=-2\sigma_2\qquad \mbox{ and }\qquad
p_4+\sigma_2p_2+4\sigma_4=0.
$$
Thus
$$
p_4=2\sigma_2^2-4\sigma_4.
$$
Consequently,
$$
\sum_{i=0}^5x_i^4-t\left(\sum_{i=0}^5x_i^2\right)^2
=
p_4-tp_2^2
=
(2-4t)\sigma_2^2-4\sigma_4.
$$    
Thus the value $t=1/2$ gives precisely the Burkhardt model $\sigma_1=\sigma_4=0$.
The balanced Pascal quartic therefore corresponds to the Burkhardt member of this
classical family.
\end{remark}

\section{The twenty Pascal planes}
\label{sec:pascal-planes}

On $X_\eta$ the twenty residual planes have the form
$$
\Pi_{ij}:\quad a+u_i=0,
\qquad a-u_j=0,
\qquad i\neq j.
$$
We identify $\Pi_{ij}$ with the directed edge $i\to j$ of the complete directed
graph on the vertex set $\{1,\ldots,5\}$.

\begin{proposition}\label{prop:line-incidence}
Two distinct Pascal planes $\Pi_{ij}$ and $\Pi_{kl}$ meet along a line if and
only if one of the following three conditions holds:
$$
i=k,
$$
$$
j=l,
$$
or
$$
(i,j)=(l,k).
$$
In all other cases the two planes meet in a single point.
\end{proposition}

\begin{proof}
The plane $\Pi_{ij}$ is defined by the two equations $a+u_i=0$ and
$a-u_j=0$ in the four-dimensional space $L_\eta$ (see Definition \ref{def:balanced}). Intersecting two such planes
gives the four equations
$$
a+u_i=0,
\quad
a-u_j=0,
\quad
a+u_k=0,
\quad
a-u_l=0
$$
in $L_\eta$. For a general pair of directed edges these equations have rank
four, and the intersection is a point. The rank drops to three precisely when
one equation is repeated or forced by the other three. This happens when the
two directed edges have the same tail, the same head, or are opposite. These
are exactly the three cases listed above.
\end{proof}
The three line incidences are explicitly
$$
\Pi_{ij}\cap\Pi_{il}=F_i\cap G_j\cap G_l,
$$
$$
\Pi_{ij}\cap\Pi_{kj}=F_i\cap F_k\cap G_j,
$$
and
$$
\Pi_{ij}\cap\Pi_{ji}=H\cap F_i\cap F_j.
$$
Consequently, there are
$$
30+30+10=70
$$
pairs of Pascal planes meeting along a line. The remaining
$$
\binom{20}{2}-70=120
$$
pairs meet in one point.

\section{The forty-five nodes}
\label{sec:nodes}

We now describe the nodes of $X_\eta$ in the Pascal coordinates. Since
$X_\eta$ is the Burkhardt quartic by Theorem~\ref{thm:balanced-is-burkhardt},
these are precisely the forty-five nodes of the Burkhardt quartic. 

There are three natural families.

\subsection{Thirty rectangular nodes}

For a partition
\begin{equation}\label{eq:AB}
\{1,\ldots,5\}=A\sqcup B\sqcup\{c\},
\qquad |A|=|B|=2,    
\end{equation}
we define a point $R_{A,B;c}$ in the affine chart $a\neq 0$ by normalizing
$a=1$ and setting
$$
u_i=-1\ \text{for } i\in A,
\qquad
u_j=1\ \text{for } j\in B,
\qquad
u_c=\eta.
$$
Equivalently,
$$
R_{A,B;c}
=
[\,1:u_1:\ldots:u_5\,].
$$
For example, for $A=\{1,2\}$ and $B=\{3,4\}$
$$
R_{\{1,2\},\{3,4\};5}=[1:-1:-1:1:1:\eta].
$$
There are
$$
5\binom{4}{2}=30
$$
such points. The point $R_{A,B;c}$ lies on exactly the four Pascal planes
$$
\Pi_{ij},
\qquad
 i\in A,
\quad
 j\in B.
$$
Thus it corresponds to an oriented rectangle in the directed complete graph.

\subsection{Ten triangular nodes}

Let $C=\{p,q\}\subset\{1,\ldots,5\}$ and put
$$
D=\{1,\ldots,5\}\setminus C.
$$
We define
$$
T_C:\quad
a=0,
\qquad
u_i=0\ \text{for } i\in D,
\qquad
u_p+u_q=0.
$$
Projectively, this means $u_p=1$ and $u_q=-1$. There are $10$ such points. The
point $T_C$ lies on exactly six Pascal planes
$$
\Pi_{ij},
\qquad
 i,j\in D,
\quad
 i\neq j.
$$
Thus it corresponds to the complete directed triangle on the three vertices of
$D$.

\subsection{Five extra nodes}

For every $c\in\{1,\ldots,5\}$ we define
$$
E_c:\quad
a=1,
\qquad
u_c=-\frac{\eta}{3},
\qquad
u_i=\frac{\eta}{3}\ \text{for } i\neq c.
$$
These five points do not lie on any of the twenty Pascal planes.

\begin{proposition}\label{prop:node-incidence-table}
The incidences between the forty-five nodes and the twenty Pascal planes are
as follows:
$$
\begin{array}{c|c|c}
\text{type of node} & \text{number of nodes} & \text{Pascal planes through the node}\\
\hline
R_{A,B;c} & 30 & 4\\
T_C & 10 & 6\\
E_c & 5 & 0
\end{array}
$$
In particular, each Pascal plane contains exactly nine nodes.
\end{proposition}

\begin{proof}
Fix $\Pi_{ij}$. A rectangular node $R_{A,B;c}$ lies on $\Pi_{ij}$ if and only
if $i\in A$ and $j\in B$. After fixing $i$ and $j$, the remaining three symbols
can be distributed by choosing one for $A$, one for $B$, and one for $c$. This
gives $3\cdot 2=6$ rectangular nodes on $\Pi_{ij}$.

A triangular node $T_C$ lies on $\Pi_{ij}$ if and only if $i,j\in D$, where
$D=\{1,\ldots,5\}\setminus C$. Equivalently, $C$ is a two-element subset of the
three symbols different from $i$ and $j$. This gives $\binom{3}{2}=3$
triangular nodes on $\Pi_{ij}$.

Finally, no $E_c$ lies on $\Pi_{ij}$, since the equations of $\Pi_{ij}$ require
$u_i=-1$ and $u_j=1$, while the coordinates of $E_c$ are all equal to
$\pm\eta/3$. Thus $\Pi_{ij}$ contains $6+3=9$ nodes.
\end{proof}

\begin{corollary}\label{cor:pascal-planes-jacobi}
The twenty Pascal planes are Jacobi planes on the Burkhardt quartic.
\end{corollary}

\begin{proof}
On the Burkhardt quartic, the planes containing nine nodes are precisely the
Jacobi planes, and they are contained in the quartic; see
\cite{BruinNasserden2018,CheltsovTschinkelZhang2025}. The assertion follows
from Proposition~\ref{prop:node-incidence-table}.
\end{proof}

\section{The complementary twenty Jacobi planes}
\label{sec:complementary-planes}

The remaining twenty Jacobi planes also have a simple form in the Pascal
coordinates. Put
$$
\zeta=\frac{-1+\eta}{2},
$$
so that $\zeta$ is a primitive third root of unity, since it satisfies
$$
\zeta^2+\zeta+1=0.
$$

Let $C=\{p,q\}$ be a two-element subset of $\{1,\ldots,5\}$, and write its
complement as an ordered triple
$$
D=\{r,s,t\}.
$$
We define two planes
$$
\Lambda_{pq}^{rst}:\quad
u_p+u_q=0,
\qquad
u_r+\zeta u_s+\zeta^2u_t=0,
$$
and
$$
\Lambda_{pq}^{rts}:\quad
u_p+u_q=0,
\qquad
u_r+\zeta^2u_s+\zeta u_t=0.
$$
A cyclic permutation of $(r,s,t)$ gives the same pair of planes, while reversing
the cyclic order interchanges $\Lambda_{pq}^{rst}$ and $\Lambda_{pq}^{rts}$.

\begin{proposition}\label{prop:lambda-contained}
The twenty planes $\Lambda_{pq}^{rst}, \Lambda_{pq}^{rts}$ are contained in $X_\eta$.
\end{proposition}

\begin{proof}
It is enough to consider $\Lambda_{pq}^{rst}$. On this plane we have
$u_p+u_q=0$ and, for the complementary triple $(r,s,t)$,
$$
u_r+\zeta u_s+\zeta^2u_t=0.
$$
Together with the ambient relation $u_1+\cdots+u_5=\eta a$, this means
$$
u_r+u_s+u_t=\eta a.$$
Substituting these two linear relations into
$$
a^4+a^2e_2(u_1,\ldots,u_5)+e_4(u_1,\ldots,u_5)
$$
and using $\eta=2\zeta+1$ and $\zeta^2+\zeta+1=0$, one obtains zero. The proof
for $\Lambda_{pq}^{rts}$ is identical.
\end{proof}

\begin{proposition}\label{prop:lambda-nodes}
Each plane $\Lambda_{pq}^{rst}$ (and $\Lambda_{pq}^{rts}$) contains exactly nine nodes of $X_\eta$.
More precisely, $\Lambda_{pq}^{rst}$ (and $\Lambda_{pq}^{rts}$) contains the three nodes
$$
E_p,
\quad
E_q,
\quad
T_{\{p,q\}},
$$
and six rectangular nodes.
\end{proposition}

\begin{proof}
The nodes $E_p$, $E_q$ and $T_{\{p,q\}}$ satisfy the defining equations of
$\Lambda_{pq}^{rst}$ (and $\Lambda_{pq}^{rts}$) directly. For the rectangular nodes, suppose first that
we are on $\Lambda_{pq}^{rst}$. The condition
$u_p+u_q=0$ forces $p$ and $q$ to have opposite signs, so one of them belongs
to $A$ and the other to $B$ (cf. \eqref{eq:AB}). The second equation is satisfied precisely when
$(u_r,u_s,u_t)$ is one of the three cyclic arrangements
$$
(-1,1,\eta),
\qquad
(1,\eta,-1),
\qquad
(\eta,-1,1).
$$
This gives $2\cdot 3=6$ rectangular nodes. The case of $\Lambda_{pq}^{rts}$ gives
the opposite cyclic arrangements. Hence each plane $\Lambda_{pq}^{rst}$ (and $\Lambda_{pq}^{rts}$)
contains $2+1+6=9$ nodes.
\end{proof}

\begin{theorem}\label{thm:forty-jacobi-planes}
The forty planes
$$
\{\Pi_{ij}:i\neq j\}
\quad\text{and}\quad
\{\Lambda_{pq}^{rst},\; \Lambda_{pq}^{rts}\; :1\leq p<q\leq 5\}
$$
are precisely the forty Jacobi planes on the Burkhardt quartic $X_\eta$.
\end{theorem}

\begin{proof}
By Corollary~\ref{cor:pascal-planes-jacobi}, the twenty Pascal planes are
Jacobi planes. By Proposition~\ref{prop:lambda-contained} and
Proposition~\ref{prop:lambda-nodes}, the twenty planes $\Lambda_{pq}^{\pm}$ are
also Jacobi planes. They are distinct from the Pascal planes, because each
$\Lambda_{pq}^{\pm}$ contains two of the extra nodes $E_c$, while no Pascal
plane contains any $E_c$. Since the Burkhardt quartic contains exactly forty
Jacobi planes, these are all of them.
\end{proof}

\begin{remark}
The visible $\mathfrak S_5$-symmetry preserves the two families of twenty
planes separately. The full automorphism group of the Burkhardt quartic is
$\textsf{PSp}_4(\mathbb F_3)$ and acts transitively on the forty Jacobi planes;
see \cite{Todd1947,CheltsovTschinkelZhang2025}. Thus individual Pascal planes
can be moved to complementary Jacobi planes by automorphisms of the Burkhardt
quartic. However, no automorphism sends the whole Pascal half to the
complementary half: with respect to the Pascal half the numbers of planes
through the three types of nodes are $4,6,0$, while with respect to the
complementary half they are $4,2,8$.
\end{remark}

\section{Steiner hyperplanes}
\label{sec:steiner}
In this section, in order to complete the combinatorial picture associated to the Burkhardt quartic, we consider the $40$ Steiner hyperplanes.

\subsection{Steiner hyperplanes from the Pascal data}
Ten Steiner hyperplanes are immediately identified as the hyperplanes used for the Pascal construction.
For every $i$ we have
$$
F_i=(a+u_i=0),
\qquad
G_i=(a-u_i=0).
$$

\begin{proposition}\label{prop:steiner}
For every $i=1,\ldots,5$,
$$
F_i\cap X_\eta=\bigcup_{j\neq i}\Pi_{ij}
$$
and
$$
G_i\cap X_\eta=\bigcup_{j\neq i}\Pi_{ji}.
$$
Moreover, each of the hyperplanes $F_i$ and $G_i$ contains exactly eighteen
nodes of $X_\eta$. Hence these ten hyperplanes are Steiner hyperplanes of the
Burkhardt quartic.
\end{proposition}

\begin{proof}
On $F_i$ we have $f_i=0$ and $g_i=h$. Therefore
$$
\frac{f_1\cdots f_5+g_1\cdots g_5}{h}\Big|_{F_i}
=
\prod_{j\neq i}g_j.
$$
This gives
$$
F_i\cap X_\eta=\bigcup_{j\neq i}(F_i\cap G_j)=\bigcup_{j\neq i}\Pi_{ij}.
$$
The proof for $G_i$ is analogous.

The hyperplane $F_i$ contains the rectangular nodes $R_{A,B;c}$ with $i\in A$.
For fixed $i$, choose the two elements of $B$ among the four remaining symbols,
and then choose $c$ among the two symbols not yet used. This gives
$$\binom{4}{2}\cdot 2=12$$
rectangular nodes. It also contains the triangular nodes $T_C$ for which
$i\notin C$, and there are $\binom{4}{2}=6$ of these. It contains no extra
node $E_c$. Hence $F_i$ contains $12+6=18$ nodes. The same argument applies to
$G_i$, with the condition $i\in B$ for rectangular nodes.

By the classical geometry of the Burkhardt quartic, hyperplanes containing
eighteen nodes are precisely Steiner hyperplanes, and each cuts the quartic in
a union of four Jacobi planes; see \cite{BruinNasserden2018,CheltsovTschinkelZhang2025}.
\end{proof}

\begin{remark}
The diagonal hyperplane $H=(a=0)$ is not one of these Steiner hyperplanes. It
contains exactly the ten triangular nodes $T_C$. The Steiner hyperplanes that
come directly from the residual construction are the ten hyperplanes $F_i$ and
$G_i$.
\end{remark}

\subsection{The additional Steiner hyperplanes}
In this part we identify geometrically the remaining $30$ Steiner hyperplanes.
Recall that
$$
u_1+\cdots+u_5=\eta a,
\qquad \eta^2=-3,
$$
and put
$$
\zeta=\frac{-1+\eta}{2}.
$$
\begin{proposition}[Steiner hyperplanes in Pascal coordinates]\label{prop:Steiner_30}
The thirty Steiner hyperplanes of the Burkhardt quartic not described in Proposition \ref{prop:steiner} are given in
Pascal coordinates by the following equations:
$$K_{pq}:\ u_p+u_q=0,
\qquad 1\leq p<q\leq 5,
$$
and
$$
S_{pq}:\ a-\zeta^2u_p+\zeta u_q=0,
\qquad p\neq q.
$$
\end{proposition}
\begin{proof}
We use the notation for the complementary Jacobi planes introduced in Section \ref{sec:complementary-planes}. 
If
$$
\{1,\ldots,5\}=\{p,q\}\sqcup\{r,s,t\},
$$
then we write
$$
\Lambda_{pq}^{rst}:
\quad
u_p+u_q=0,\qquad
u_r+\zeta u_s+\zeta^2u_t=0.
$$
Cyclic permutation of the ordered triple $(r,s,t)$ gives the same plane,
whereas reversing the cyclic order gives the conjugate plane.

We first observe that each hyperplane in the Proposition contains four Jacobi planes.

To begin with, the hyperplane
$$
K_{pq}=\{u_p+u_q=0\}.
$$
contains $\Pi_{pq}$ and $\Pi_{qp}$, since on
$\Pi_{pq}$ one has $u_p=-a$ and $u_q=a$, and similarly for $\Pi_{qp}$.
It also contains the two complementary Jacobi planes
$$
\Lambda_{pq}^{rst}
\quad\text{and}\quad
\Lambda_{pq}^{rts},
$$
where $\{r,s,t\}=\{1,\ldots,5\}\setminus\{p,q\}$. Therefore
$$
X_\eta\cap K_{pq}
=
\Pi_{pq}+\Pi_{qp}+\Lambda_{pq}^{rst}+\Lambda_{pq}^{rts}.
$$
Equivalently, this can be checked by factoring the equation on $K_{pq}$.
Indeed, on $K_{pq}$ we have $u_q=-u_p$, and the relation
$$
u_1+\cdots+u_5=\eta a
$$
becomes
$$
u_r+u_s+u_t=\eta a.
$$
Hence
$$
a^2+e_2(u_r,u_s,u_t)
=
-\frac{1}{3}
\left(u_r+\zeta u_s+\zeta^2u_t\right)
\left(u_r+\zeta^2u_s+\zeta u_t\right),
$$
because $\eta^2=-3$. Thus the restriction of the quartic to $K_{pq}$ factors,
up to a non-zero scalar, as
$$
(a+u_p)(a-u_p)
\left(u_r+\zeta u_s+\zeta^2u_t\right)
\left(u_r+\zeta^2u_s+\zeta u_t\right).
$$
These four factors give exactly the four planes displayed above.

It remains to consider the hyperplanes
$$
S_{pq}=\{a-\zeta^2u_p+\zeta u_q=0\},
\qquad p\neq q.
$$
First, $S_{pq}$ contains the Pascal plane $\Pi_{pq}$, because on $\Pi_{pq}$
we have $u_p=-a$ and $u_q=a$, and hence
$$
a-\zeta^2u_p+\zeta u_q
=
a+\zeta^2a+\zeta a
=
(1+\zeta+\zeta^2)a
=
0.
$$
Now let $\{r,s,t\}=\{1,\ldots,5\}\setminus\{p,q\}$. We claim that
$S_{pq}$ contains the three complementary Jacobi planes
$$
\Lambda_{rs}^{pqt},\qquad
\Lambda_{rt}^{pqs},\qquad
\Lambda_{st}^{pqr}.
$$
For example, on $\Lambda_{rs}^{pqt}$ we have
$$
u_r+u_s=0
$$
and
$$
u_p+\zeta u_q+\zeta^2u_t=0.
$$
The second equation gives
$$
u_t=-\zeta u_p-\zeta^2u_q.
$$
Using the relation $u_1+\cdots+u_5=\eta a$ and the equality
$u_r+u_s=0$, we obtain
$$
u_p+u_q+u_t=\eta a.
$$
Substituting the expression for $u_t$ gives
$$
(1-\zeta)u_p+(1-\zeta^2)u_q=\eta a.
$$
Since $\eta=\zeta-\zeta^2$, this is equivalent to
$$
a=\zeta^2u_p-\zeta u_q.
$$
Thus
$$
a-\zeta^2u_p+\zeta u_q=0,
$$
so $\Lambda_{rs}^{pqt}\subset S_{pq}$. The same argument applies to
$\Lambda_{rt}^{pqs}$ and $\Lambda_{st}^{pqr}$.

Therefore
$$
X_\eta\cap S_{pq}
=
\Pi_{pq}
+
\Lambda_{rs}^{pqt}
+
\Lambda_{rt}^{pqs}
+
\Lambda_{st}^{pqr}.
$$

We have thus found
$$
\binom{5}{2}+5\cdot 4=30
$$
distinct hyperplanes, and each cuts the Burkhardt quartic into a union of four
Jacobi planes. 
\end{proof}

\begin{corollary}
The hyperplanes listed in Propositions~\ref{prop:steiner} and
\ref{prop:Steiner_30} are precisely the forty Steiner hyperplanes of the
Burkhardt quartic.
\end{corollary}

\begin{proof}
Combining Propositions~\ref{prop:steiner} and~\ref{prop:Steiner_30}, we have
found altogether forty hyperplanes whose sections by $X_\eta$ are unions of
four Jacobi planes. By the classical description of the Burkhardt quartic,
there are exactly forty Steiner hyperplanes. Hence the listed hyperplanes are
all of them.
\end{proof}

\section{Further questions}
\label{sec:further-questions}

We conclude with some questions suggested by the constructions considered here.

\begin{question}
Which hyperplane sections
$$
X_\lambda=\mathcal X\cap L_\lambda
$$
are nodal? More precisely, what is the discriminant in the parameter space
$(\PP^5)^\vee$, and which special points or subvarieties of this discriminant
have a classical interpretation?
\end{question}

\begin{question}
Is there a residual or Pascal-type construction that produces the complementary
Jacobi planes $\Lambda_{pq}^{rst}$ directly, rather than through the
identification with the Burkhardt quartic?
\end{question}

\begin{question}
How does the Pascal decomposition
$$
40=20+20
$$
of the Jacobi planes interact with the classical modular interpretation of the
Burkhardt quartic as a compactification of a Siegel modular threefold?
\end{question}

\paragraph{Acknowledgements.} We would like to thank Igor Dolgachev for his comments on the first draft of this work.

%\bibliographystyle{abbrv}
%\bibliography{pascal-type}

\bigskip
\noindent
Tomasz Szemberg, Justyna Szpond\\
Department of Mathematics,\\
University of the National Education Commission Krakow,\\
Podchor\c a\.zych 2,
PL-30-084 Krak\'ow, Poland. \\
\nopagebreak
\textit{E-mail address:} \\
\texttt{tomasz.szemberg@gmail.com}\\
\texttt{szpond@gmail.com}
\end{document}